\documentclass[preprint]{imsart}

\RequirePackage{amsthm,amsmath,amsfonts,amssymb}
\RequirePackage[numbers]{natbib}
\RequirePackage[colorlinks,citecolor=blue,urlcolor=blue]{hyperref}
\RequirePackage{graphicx}
\RequirePackage{bbm,dsfont}
\RequirePackage{comment}

\startlocaldefs

 \newtheorem{thm}{Theorem}[section]
 \newtheorem{theorem}{Theorem}[section]
 \newtheorem{corollary}[thm]{Corollary}
 \newtheorem{lemma}[thm]{Lemma}

 \theoremstyle{definition}
 
 \newtheorem{remark}[thm]{Remark}
 
 \newtheorem{assumption}[thm]{Assumption}

 \newcommand{\ind}{\mathbbm{1}}

 \newcommand{\R}{\mathbb{R}}
 \newcommand{\Z}{\mathbb{Z}}

 \newcommand{\ic}{\mathrm{i}}

 \newcommand{\E}{\mathbb{E}}

 \newcommand{\N}{\mathbb{N}}

 \newcommand{\oo}{\mbox{\scriptsize $\mathcal{O}$}}

 \newcommand{\rd}{\mathfrak{r}}

 \newcommand{\FF}{\mathcal{F}}
 \newcommand{\Td}{\mathfrak{T}}
 \newcommand{\Cd}{\mathfrak{C}}
 \newcommand{\kapp}{\kappa}

\renewcommand{\P}{\mathbb{P}}
\renewcommand{\ss}{\mathfrak{s}}

\newcommand{\Aone}{{\bf (A1)}}
\newcommand{\Atwo}{{\bf (A2)}}
\newcommand{\Athree}{{\bf (A3)}}

\newcommand{\Vone}{{\bf (V1)}}
\newcommand{\Vtwo}{{\bf (V2)}}

\newcommand{\Bone}{{\bf (B1)}}
\newcommand{\Btwo}{{\bf (B2)}}

\endlocaldefs

\begin{document}

\begin{frontmatter}
\title{Edgeworth expansions for volatility models}
\runtitle{Edgeworth expansions for volatility models}

\begin{aug}
\author[A]{\fnms{Moritz} \snm{Jirak}\ead[label=e1]{moritz.jirak@univie.ac.at}},
\address[A]{University of Vienna.
\printead{e1}}

\end{aug}

\begin{abstract}
Motivated from option and derivative pricing, this note develops Edgeworth expansions both in the Kolmogorov and Wasserstein metric for many different types of discrete time volatility models and their possible transformations. This includes, among others, H\"{o}lder-type functions of (augmented) Garch processes of any order, iterated random functions or Volterra-processes.
\end{abstract}


\end{frontmatter}

\section{Introduction}

Consider a strictly stationary sequence $(X_k)_{k \in \Z}$ of real-valued random variables with $\E X_k = 0$ and $\E X_k^2 < \infty$. If the sequence exhibits weak dependence in a certain sense, then the distribution \hypertarget{sn:eq2}{of}
\begin{align*}
n^{-1/2}S_n, \,\, \text{where $S_n = X_1 + X_2 + \ldots + X_n$},
\end{align*}
is asymptotically normal, see for instance \cite{maxwellwoodroofe} and the references therein. This fact has made the central limit theorem one of the most important tools in probability theory and statistics. On the other hand, it was already noticed by Chebyshev \cite{chebyshev1890} and Edgeworth \cite{Edgeworth1894} that normal approximations can be improved in terms \hypertarget{psin:eq2}{of} (Edgeworth) expansions $\Psi_n$, implying the approximation
\begin{align}\label{eq:edge}
\sup_{x \in \R}\big|\P\big(S_n \leq x\sqrt{n}\big) - \Psi_n(x)\big| = \oo\big(n^{-\frac{1}{2}}\big) \quad \text{(or even better)}
\end{align}
in the Kolmogorov metric, where
\begin{align}\label{defn_PSI_m}
\Psi_n\bigl(x\bigr) = \Phi\bigl(x/s_{n}\bigr) + {\kappa}_{n}^3\bigl(1 - x^2/{s}_{n}^2 \bigr) \phi\bigl(x/{s}_{n}\bigr)
\end{align}
with
\begin{align}\label{defn:sn:kappan}
s_n^2 = n^{-1}\E S_n^2, \quad \kappa_n^3 = n^{-3/2}\E S_n^3.
\end{align}

Motivated by applications in actuarial science, Cram\'{e}r gave rigorous proofs in ~\cite{cramer1928}, and ever since, Edgeworth expansions have been an indispensable tool in actuarial science and finance, see the discussion below. They also arise in the context of dynamical system theory and Markovian setups, e.g.~\cite{kasun:2018:edgeworth},~\cite{meyn_kontoyiannis_2003aap},~\cite{herve_pene_2010_bulletin} and the references therein for more recent results, and also~\cite{goetze_hipp_1983},~\cite{lahiri_1993} for a general, weakly dependent framework\footnote{These lists of references are by no means complete and only represent a mere fraction of the vast literature in this area.}.
On the other hand, in a very influential work, Efron~\cite{efron1979} broadened the view on resampling techniques (e.g. bootstrapping) and demonstrated their significant superior performance compared to normal approximations, see~\cite{efron_1982},~\cite{hall_book_boot},~\cite{lahiri_book_2003} for an overview. Not surprisingly, the key tools for analysing, and, in particular, showing superiority of resampling methods, are again Edgeworth expansions.\\
\\
Our main motivation here stems more from actuarial, econometric, finance and risk management considerations. Very prominent models in these areas in a discrete time setting are (augmented) Garch processes, e.g.~\cite{duan:1995:garch},~\cite{duan_1997}, ~\cite{bollerslev_garch},~\cite{heston:garch:2000},~\cite{garch:book:2010}. It is well-known that already the Black-Scholes formula for option pricing has serious shortcomings, e.g.~\cite{gatheral}. To address these problems, a common and quite successful approach is to employ more complex models and use (Edgeworth) expansions to salvage comparatively simple and easy to evaluate formulas, see for instance~\cite{alos_2006_fin_and_stoch},~\cite{bayer:2017:qf},~\cite{duan:garch:2000:edge},~\cite{fouque:2003},~\cite{friz-AAP1608},~\cite{Fukasawa_2011ejp},~\cite{hansen:garch:2017}. To illustrate this further, consider the (standard) model
\begin{align}
\log P_{nt} = \sum_{k = 1}^{nt} V_{k-1} \frac{\epsilon_k}{\sqrt{n}} - \frac{1}{2n} \sum_{k = 1}^{nt} h(V_{k-1}^2), \quad t \in [0,1],
\end{align}
where $\log P_{nt}$ is the log-price of some derivative under a martingale measure for appropriate function $h$. Here, $\epsilon_k$ are the innovations and $V_k^2$ is some volatility process with $V_k^2 \in \sigma(\epsilon_{k}, \epsilon_{k-1}, \ldots)$. Then one seeks an approximation of the type
\begin{align*}
\E f\big(\log P_{nt} \big) \approx \int f(x) \Lambda_{nt}\bigl(dx\bigr),
\end{align*}
where the function $f(x)$ describes the pay-off of some option and $\Lambda_{nt}$ is a 'convenient', signed measure. Since we may express $\log P_{nt} = S_{nt}/\sqrt{n} + \sum_{k = 1}^{nt} \E h(V_{k-1}^2) /2n$ in terms of a normalised and centred sum of (dependent) random variables, the connection to Edgeworth expansions is obvious. A rather prominent example in this context are European Put-Options, where $f(x) =  \max\{K-e^x,0\}$ for some strike price $K > 0$. Now $|f(x) - f(y)| \leq K|x-y|$, and hence $f(x)$ is Lipschitz-continuous. The latter is true for many options, and thus another natural metric to measure the quality of Edgeworth expansions is in terms of the Wasserstein metric $W_1$. In the latter, a Gamma-type approximation is more convenient in our setting, see \eqref{wasserstein_1} for more details and definitions.\\
\\
Our main contribution and novelty here is to establish the validity of Edgeworth expansions both in the Kolmogorov and Wasserstein metric for various classes of popular volatility type models. This includes in particular - for the first time, to the best of the authors knowledge - functions of (augmented) Garch($\mathfrak{p},\mathfrak{q}$)-processes of any order. Previously, only the case $\mathfrak{p} = \mathfrak{q} = 1$ appears to have been treated in the literature. In addition, it seems that there are almost no results concerning Edgeworth expansions for weakly dependent processes in terms of the Wasserstein distance in general.\\
\\
This note is structured as follows. In Section \ref{sec:main}, we present the setup and our main global results. We then show how to use these to derive the validity of Edgeworth expansions both in the Kolmogorov and Wasserstein metric for H\"{o}lder type functionals of various volatility type models, see Section \ref{sec:aug:garch} (augmented Garch processes), Section \ref{sec:iterated} (iterated random functions), Section \ref{sec:linear} (linear processes) and Section \ref{sec:volterra} (Volterra processes) for details.

\section{Setup and main global results}\label{sec:main}

{For} a random variable $X$, we write $\E X $ for expectation, $\|X\|_p$ for $\big(\E |X|^p \big)^{1/p}$, $p \geq 1$, and sometimes $\E_{\mathcal{H}} X = \E [X | \mathcal{H}]$ for conditional expectation, and in analogy $\P_{\mathcal{H}}(\cdot)$ for conditional probabilities. $\lesssim$, $\gtrsim$, ($\approx$) denote (two-sided) inequalities involving a multiplicative constant. For $a,b \in \R$, we put $a \vee b = \max\{a,b\}$, $a \wedge b = \min\{a,b\}$. For two random variables $X,Y$, we write $X \stackrel{d}{=} Y$ for equality in distribution. For an i.i.d. sequence $(\epsilon_k)_{k \in \Z}$, let $\mathcal{E}_k = \sigma\bigl(\epsilon_j, \, j \leq k\bigr)$, and $\mathcal{E}_k^+ = \sigma\bigl(\epsilon_j,\, j \geq k\bigr)$.

Consider a sequence of real-valued, measurable random variables $X_1,\ldots,X_n$. It is well known (cf.~\cite{rosenblatt:book}), that this sequence can be assumed to satisfy $X_k \in \mathcal{E}_k$, that is, we have

\begin{align}\label{eq_structure_condition}
{X}_{k} = g_k\bigl(\epsilon_{k}, \epsilon_{k-1}, \ldots \bigr)
\end{align}
for some measurable functions $g_k$\footnote{In fact, $g_k$ can be selected as a map from $\R^k$ to $\R$}, where $(\epsilon_k)_{k \in \Z}$ is a sequence of independent and identically distributed random variables. For notational
convenience, we sometimes assume $g_k = g$, that is, the function $g$ does not depend on $k$. Such processes are usually referred to as (time-homogenous) Bernoulli-shift processes. 

Representation \eqref{eq_structure_condition} allows to give simple, yet very efficient and general dependence conditions. Following \cite{wu_2005}, let $(\epsilon_k')_{k \in \Z}$ be an independent copy of $(\epsilon_k)_{k \in \Z}$ on the same probability space, and define the 'filter' $\theta_{k}^{(l,*)}$ as
\begin{align}\label{defn_strich_depe_2}
\theta_{k}^{(l,*)} = \bigl(\epsilon_{k}, \epsilon_{k - 1},\ldots,\epsilon_{k - l}',\epsilon_{k - l - 1}',\epsilon_{k-l-2}',\ldots\bigr).
\end{align}
We write
\begin{align*}
\theta_{k}^* = \theta_{k}^{(k,*)} = (\epsilon_{k }, \epsilon_{k - 1},\ldots,\epsilon_{0}',\epsilon_{- 1}',\epsilon_{- 2}',\ldots ),
\end{align*}
${X}_{k}^{(l,*)} = g_k(\theta_{k}^{(l,*)})$, and ${X}_{k}^{*} = {X}_{k}^{(k,*)}$. As dependence measure, one may then define
\begin{align}\label{defn:dep:measure:general}
\vartheta_{l}^{\ast}(p) = \sup_{k \in\mathbb{Z}} \|{X}_{k} - {X}_{k}^{(l,\ast)}\|_p.
\end{align}
If $g = g_k$ does not depend on $k$ (time-homogenous case), $\vartheta_{l}^{\ast}(p)$ simplifies to
\begin{align*}
\vartheta_{l}^{\ast}(p)=\big\|X_l-X_l^*\big\|_p.
\end{align*}
Our basic condition regarding weak dependence is now the following.

\begin{assumption}\label{ass:dependence:main}
For $p > 3$, $(X_k)_{k\in \Z}$ is stationary and satisfies
\begin{enumerate}
\item[\Aone]\label{A1} $\E |X_k|^p < \infty$, $\E X_k = 0$,
\item[\Atwo]\label{A2} $\sum_{k = 1}^{\infty}k^{2}\vartheta_{k}^{\ast}(p) < \infty$,
\item[\Athree]\label{A3} $\ss^2 > 0$, where $\ss^2 =  \sum_{k \in \Z}\E X_0X_k$.
\end{enumerate}
\end{assumption}

Our requirement of stationarity is more a convenience condition, and can be replaced with quenched or locally stationary setups.

As is well known, validity of Edgeworth expansions is not for free and requires some non-lattice condition. We need the following regularity asssumptions regarding the underlying distribution.

\begin{assumption}\label{ass:smallball:I}
Consider $\mathcal{E}_l' = \sigma(\epsilon_k', k \leq l)$, which is independent of $\mathcal{E}_l$. For any $\delta > 0$ and $l \in \Z$, there exists a family of random variables  $(X_k^+) \in \mathcal{G}_l \stackrel{def}{=} \sigma(\mathcal{E}_l^+ \cup \mathcal{E}_l')$, such that
\begin{enumerate}
\item[\Bone]\label{B1} $\P\Big(\sum_{k \geq l} \E_{\mathcal{G}_l}\big|X_k - X_k^+\big| \leq \delta  \Big) > 0$,
\item[\Btwo]\label{B2} $\E\big|\E_{\mathcal{E}_{l-1}} e^{\mathrm{i} \xi X_l }  \big| < 1$ for any $\xi \neq 0$.
\end{enumerate}
\end{assumption}


Observe that condition \hyperref[B2]{\Btwo} is a non-lattice condition, and will be easy to verify in case of our applications. The key to our results is \hyperref[B1]{\Bone}, which is a small ball condition. While it is not true in general, we show below that it does hold for a huge class of volatility models and their Hölder-continuous transformations (and even more). Note that validity of \hyperref[B1]{\Bone} does not imply that $\sum_{k \geq l} \E_{\mathcal{E}_l^+}\big|X_k - X_k^+\big|$ or even $X_k$ is non-lattice. Our first result is the following.

\begin{theorem}\label{thm:edge:smallball}
Assume that Assumptions \ref{ass:dependence:main} and \ref{ass:smallball:I} hold. Then
\begin{align*}
\sup_{x \in \R}\bigl|\P\bigl(S_n \leq x \sqrt{n}\bigr) -  \Psi_n\bigl(x\bigr) \bigr| = \oo\bigl(n^{-1/2}\bigr).
\end{align*}
In particular, there exists $b_n \to \infty$ and $\delta > 0$ such that for any $a > 0$
\begin{align}\label{eq:thm:edge:smallball:condition}
\sup_{\xi \in [a,b_n]} \Big|\E e^{\ic \xi {S}_{n}} \Big| \leq C_{a,\delta}  n^{-1/2 - \delta}.
\end{align}
\end{theorem}

Next, we turn to the Wasserstein metric $W_1$. For two probability measures $\P_1,\P_2$, let $\mathcal{L}(\P_1,\P_2)$ be the set of all probability measures on $\R^2$ with marginals $\P_1, \P_2$. The Wasserstein metric (of order one) is defined as the minimal coupling $\mathds{L}^1$-distance, \hypertarget{Wone:eq10}{that} is,
\begin{align}\label{wasserstein_1}
W_1(\P_1,\P_2) = \inf\Big\{ \int_{\R^2} |x-y| \P(dx, dy): \, \P \in \mathcal{L}(\P_1, \P_2)\Big\}.
\end{align}


Let $\mathds{V}_n$ be the (signed) measure induced by $\Psi_n$. Then a priori, the distance $W_1(\P_1, \mathds{V}_n)$ is not defined in general. In ~\cite{Bobkov2017_wasser}, generalized transport distances are introduced that also allow for signed measures. In order to maintain the original definition in terms of couplings, we follow ~\cite{jirak:edge:trans:2021} and replace $\Psi_n$ with a probability measure that is induced by a sequence of i.i.d. random variables. Let $Z$ be a \hypertarget{kappan3:eq10}{zero} mean Gaussian random variable $\mathcal{N}(0,\sigma^2)$ \hypertarget{TAUSUBN:eq11}{with} variance $\sigma^2 = s_n^2$, and $G$ follow a Gamma distribution $\Gamma(\alpha, \beta)$ with shape parameter $\alpha = s_n^2 \beta$ and rate $\beta = 2 s_n^{2} (\sqrt{n} \kapp_n)^{-3}$, independent of $Z$ (recall \eqref{defn:sn:kappan} for $s_n$, $\kappa_n$). 
 \hypertarget{bigMk:eq10}{Then} for $1\leq k \leq n$, \hypertarget{MMn:eq10}{let}
\begin{align}\label{defn_gauss_plus_gamma}
M_k \stackrel{d}{=}  \left\{
\begin{array}{ll}
\big(Z + G-\E G\big)/\sqrt{2} &\text{if $\kapp_n^3 > 0$},\\
\big(Z - G+\E G\big)/\sqrt{2} &\text{if $\kapp_n^3 < 0$},\\
Z &\text{if $\kapp_n^3 = 0$}
\end{array}
\right.
\end{align}
be i.i.d., and denote by $\P_{L_n}$ the probability measure induced by $L_n = n^{-1/2}\sum_{k = 1}^n M_k$. Observe that $\E L_n^2 = s_n^2$ and $\E L_n^3 = \kapp_n^3$. Also note that $L_n \stackrel{d}{=} Z_0 \pm  (G_0 - \E G_0)$, with $Z_0,G_0$ independent and $Z_0 \stackrel{d}{=}\mathcal{N}(0,s_n^2)$, $G_0 \stackrel{d}{=}\Gamma(n \alpha, \beta/\sqrt{n})$ with $\alpha$, $\beta$ as above.

\begin{theorem}\label{thm:wasser:general}
Grant Assumption \ref{ass:dependence:main}, and suppose that \eqref{eq:thm:edge:smallball:condition} holds.
Then
\begin{align*}
W_1\big(\P_{S_n^{}/\sqrt{n}}, \P_{L_n}\big) = o\big(n^{-1/2}\big).
\end{align*}
\end{theorem}

Due to Theorem \ref{thm:edge:smallball}, an immediate consequence is the following.

\begin{corollary}\label{cor:wasser:small:ball}
Grant Assumptions \ref{ass:dependence:main} and \ref{ass:smallball:I}. Then
\begin{align*}
W_1\big(\P_{S_n^{}/\sqrt{n}}, \P_{L_n}\big) = o\big(n^{-1/2}\big).
\end{align*}
\end{corollary}


\section{Volatility models}\label{sec:vola}

Over the past decades, the following basic model has emerged as a key building block in econometrics, finance and actuarial science for an underlying process $(Y_k)$:

\begin{align}\label{eq:vola:model}
Y_k = \epsilon_k V_{k-1}, \quad k \in \Z,
\end{align}
where $(\epsilon_k)_{k \in \Z}$ are i.i.d. and $V_k \in \mathcal{E}_k$ is some volatility process. The actual models for asset prices (and related) are then obtained by appropriate transformations, compensations or by passing on to the limit to obtain stochastic differential equations. A sheer endless amount of models and processes of this type have been established and discussed. Since our focus here lies on discrete time, we mention for instance~\cite{duan:1995:garch},~\cite{duan_1997},~\cite{bollerslev_garch},~\cite{heston:garch:2000},~\cite{garch:book:2010} which, however, presents only an almost infinitesimal fraction of the literature.\\
\\
Our basic setup here is the following. We consider processes $Y_k$ of type \eqref{eq:vola:model}, where we assume that $V_k \in \mathcal{E}_k$ is stationary. In the sequel, it will be convenient to use $\varepsilon \stackrel{d}{=} \epsilon_k$. Let $X_k = f(Y_k) + h_n(V_{k-1})$ ($f,h_n$ specified below) and $S_k = \sum_{j = 1}^k X_j$. Consider the price of an asset $P_k = e^{n^{-1/2}S_k}$. If we select $h_n \equiv 0$ and $f$ such that
\begin{align}\label{f:expansion}
f(x) = x -\frac{x^2}{2\sqrt{n}} - \frac{2x^3}{3 n} + r(x), \quad  r(Y_k) = O_{\P}(n^{-3/2}),
\end{align}
then, given sufficiently many (exponential) moments\footnote{Using smooth truncation functions, moment conditions can be drastically reduced here while maintaining the result.}, we obtain
\begin{align}\label{P:almost:martingale}
\E\big[P_{k+1} \big| \mathcal{E}_k \big] = P_k + O_{\P}\big(n^{-2}\big).
\end{align}
Hence $P_k$ is almost a martingale, and the actual error in \eqref{P:almost:martingale} can be made arbitrarily small by further specifying $g$. On the other hand, formally letting
\begin{align}
X_k = Y_k - h_n(V_{k-1}), \quad h_n(V_{k-1}) = n^{1/2}\log \E\big[e^{n^{-1/2} Y_k}|V_{k-1}\big],
\end{align}
it follows that $P_k$ is a martingale. Note in particular that if $\varepsilon$ has a standard Gaussian distribution $\mathcal{N}(0,1)$, then we get the well-known form $X_k = Y_k - V_{k-1}^2/ (2 n^{1/2})$. More generally, a formal Taylor expansion around zero with $\E \varepsilon = 0$ leads to
\begin{align}\label{eq:formal:taylor:vola}
h_n(V_{k-1}) = \frac{V_{k-1}^2}{2 \sqrt{n}} \E \varepsilon^2 + O_{\P}(n^{-1}).
\end{align}

We wish to apply Theorem \ref{thm:edge:smallball} to $X_k = f(Y_k) + h_n(V_{k-1})$, where we assume $\E X_k = 0$. Having in mind \eqref{f:expansion}, \eqref{eq:formal:taylor:vola}, but also statistical applications (power transformations), we consider functions $f,h_n$ satisfying the generalised H\"{o}lder condition
\begin{align}\label{eq:hoelder:function}
\big|f(x) - f(y)\big| \leq L |x-y|^{\beta}\big(1 + |x|^{\alpha} + |y|^{\alpha}\big), \quad \alpha \geq 0,\, \beta,L > 0.
\end{align}
For future reference, we denote this class with $\mathcal{H}(L,\alpha,\beta)$. Moreover, we assume that
\begin{align}\label{eq:h:limit:zero}
\lim_{n \to \infty} \E \big|h_n(V_{k}) \big| = 0,
\end{align}
which in light of \eqref{eq:formal:taylor:vola} is a mild condition.

Our basic condition to verify Assumption \ref{ass:smallball:I} is the following.

\begin{assumption}\label{ass:smallball:vola}
Assume that $Y_k$ is stationary, is of the form \eqref{eq:vola:model}, the functions $f,h_n$ satisfy \eqref{eq:hoelder:function} (uniformly in $n$), $\ss^2 =  \sum_{k \in \Z}\E X_0X_k > 0$, and in addition:
\begin{enumerate}
\item[\Vone]\label{V1} For any $\delta > 0$, we have $\P\big(|\varepsilon| \leq \delta  \big) > 0$.
\item[\Vtwo]\label{V2} There exists a compact set $\mathcal{V}$ with $\P(V_0 \in \mathcal{V}) > 0$, such that for any $v \in \mathcal{V}$ and $\xi \neq 0$, we have $\big|\E e^{\mathrm{i} \xi f(\varepsilon v)} \big| < 1$.
\end{enumerate}
\end{assumption}

\begin{remark}
If we strengthen \hyperref[V1]{\Vone} a little bit, we can simplify condition \hyperref[V2]{\Vtwo}: Suppose that $\P( |\varepsilon| \leq \delta, \varepsilon \neq 0 ) > 0$ for all $\delta > 0$. Then it is sufficient to demand that $f$ is non-constant in a surrounding of zero. 
\end{remark}

We now discuss three particular, yet quite general class of models. We note, however, that the method of proof is more general and can also be applied to other classes of models.

\subsection{Augmented Garch}\label{sec:aug:garch}

Arch, Garch and augmented Garch models have had a huge impact both in theory and practice, see~\cite{bollerslev_garch},~\cite{duan_1997} and~\cite{garch:book:2010}. In one of its more general forms, it can be stated as

\begin{align}
\Lambda(V_k^{2}) = \sum_{i = 1}^{\mathfrak{p}} g_i(\epsilon_{k-i}) + \sum_{i = 1}^{\mathfrak{q}} c_i(\epsilon_{k-i}) \Lambda(V_{k-i}^2),
\end{align}
where $g_i, c_i \geq 0$ are positive functions, and $g_{i_0} \geq \omega > 0$ for some $1 \leq i_0 \leq \mathfrak{p}$. Motivated from the Box-Cox transformation, the function $\Lambda(\cdot)$ typically comprises the cases $\Lambda(x) = \log x$ or $\Lambda(x) = x^{\lambda}$, $\lambda > 0$. We will consider the latter. For $q \geq 1$, an important quantity is
\begin{align*}
\gamma_c = \sum_{i = 1}^{\mathfrak{r}} \|c_i(\varepsilon_0)\|_{q}, \quad \text{with $\mathfrak{r} = \max\{\mathfrak{p},\mathfrak{q}\}$},
\end{align*}
where we replace possible undefined $c_i$ (and $g_i$) with zero. If $\gamma_c < 1$, then $(V_k)_{k \in \Z}$ is stationary. In particular, one can show the representation
\begin{align*}
V_k^{2\lambda} = \sum_{m = 1}^{\infty} \sum_{1 \leq l_1,\ldots,l_m\leq \mathfrak{r}}  g_{l_m}(\varepsilon_{k - l_1 - \ldots - l_m}) \prod_{i = 1}^{m-1} c_i(\varepsilon_{k - l_1 - \ldots - l_i}),
\end{align*}
see~\cite{berkes_2008_letter} for comments and references on this matter. In particular, $V_k$ is a time-homogenous Bernoulli-shift process, that is, $g_k = g$ does not depend on $k$ in representation \eqref{eq_structure_condition}.

In the sequel, we will only treat the case $\lambda \geq 1/2$, the case $\lambda < 1/2$ follows in essentially the same manner,
requiring more moment conditions though. Our assumptions are now the following.

\begin{assumption}\label{ass:garch}
Let $q > \max\{(\alpha \vee \beta )/\lambda, 3(\alpha + \beta)\}$ ($\alpha, \beta$ as in \eqref{eq:hoelder:function}), $\lambda \geq 1/2$. For all $1 \leq i \leq \mathfrak{r}$, there exists $\delta > 0$ such that $\sup_{|x| \leq \delta} c_i(x) \leq \|c_i(\varepsilon_0)\|_{q}$ and $\sup_{|x| \leq \delta} g_i(x) \leq C_g$. Moreover, we have $\gamma_c < 1$ and $\|g_i(\varepsilon)\|_q, \E |\varepsilon|^{2 \vee q} < \infty$.
\end{assumption}

We then have the following result.

\begin{theorem}\label{thm:garch}
Grant Assumptions \ref{ass:smallball:vola} and \ref{ass:garch}. Then both Assumptions \ref{ass:dependence:main} and \ref{ass:smallball:I} hold. In particular, Theorem \ref{thm:edge:smallball} and Corollary \ref{cor:wasser:small:ball} apply.
\end{theorem}

\subsection{Iterated random functions}\label{sec:iterated}

An iterated random function system on the state space $\R$ is defined as
\begin{align}
V_k = F_{\epsilon_k}\big(V_{k-1} \big), \quad k \in \N,
\end{align}
where $\epsilon_k \in \R$ are i.i.d. with $\varepsilon \stackrel{d}{=} \epsilon_k$. Here, $F_{\varepsilon}(\cdot) = F(\cdot, \varepsilon)$ is the $\varepsilon$-section of a jointly
measurable function $F : \R \times \R \to \R$. Many dynamical systems, Markov processes and non-linear time series are within this framework, see for instance ~\cite{diaconis_freedman_1999}. For $y \in \R$, let
$V_k(y) = F_{\epsilon_k} \circ F_{\epsilon_{k-1}} \circ \ldots \circ F_{\epsilon_0}(y)$, and, given $y,y' \in \mathcal{Y}$ and $\gamma > 0$, we say that the system is $\gamma$-\textit{moment contracting} if
\begin{align}\label{ex:iterated:contraction}
\sup_{y,y'}\E \big|V_k(y) - V_k(y')\big|^{\gamma} \leq C \rho^k, \quad \rho \in (0,1).
\end{align}
We note that some variations exist in the literature. A key quantity for verifying the moment contraction \eqref{ex:iterated:contraction} is
\begin{align*}
L_{\varepsilon} = \sup_{y \neq y'}\frac{\big|F_{\varepsilon}(y)- F_{\varepsilon}(y')\big|}{|y - y'|}.
\end{align*}
Essentially (subject to some mild regularity conditions), \eqref{ex:iterated:contraction} holds if $\E L_{\varepsilon}^{\gamma} < \infty$ and $\E \log L_{\varepsilon} < 0$, see ~\cite{diaconis_freedman_1999}, ~\cite{wu_shao_iterated_2004}. Note that \eqref{ex:iterated:contraction} implies $\vartheta_{k}^{\ast}(\gamma) \leq C \rho^k$ for $\gamma \geq 1$. 

\begin{assumption}\label{ass:iterated}
Let $q > \max\{3(\alpha + \beta),1\}$ ($\alpha, \beta$ as in \eqref{eq:hoelder:function}). There exists $\delta > 0$, such that $\sup_{|\varepsilon|\leq \delta} L_{\varepsilon} < 1$. Moreover, we have $\|L_{\varepsilon}\|_q < 1$ and $\E |F_{\varepsilon}(0)|^q, \E |\varepsilon|^{2 \vee q} < \infty$.
\end{assumption}

We then have the following result.
\begin{theorem}\label{thm:iterated}
Grant Assumptions \ref{ass:smallball:vola} and \ref{ass:iterated}. Then both Assumptions \ref{ass:dependence:main} and \ref{ass:smallball:I} hold. In particular, Theorem \ref{thm:edge:smallball} and Corollary \ref{cor:wasser:small:ball} apply.
\end{theorem}

\subsection{Functions of linear processes}\label{sec:linear}

Linear processes and their transformations are key in time series analysis, see e.g.~\cite{timeseriesbrockwell}.
For $\alpha_k \in \R$ and real-valued functions $c_i(\cdot)$, we formally define the linear process
\begin{align*}
G_k = \sum_{i = 0}^{\infty} a_i c_i(\epsilon_{k-i}), \quad k \in \Z.
\end{align*}
Let $g \in H(L, \gamma,\delta)$, and consider
\begin{align*}
V_k = g\bigl(G_k\bigr), \quad k \in \Z.
\end{align*}
The following is our main assumption.

\begin{assumption}\label{ass:linear}
Let $q > \max\{3(\alpha + \beta)(\gamma+\delta), 1\}$ ($\alpha, \beta$ as in \eqref{eq:hoelder:function}) and suppose that $\sup_{i}\|c_i(\varepsilon)\|_q < \infty$, $\sum_{i \geq 1} i^2 |a_i|^{\beta \delta} < \infty$ and $\E |\varepsilon|^{2(\alpha \vee 1)} < \infty$.
\end{assumption}

We then have the following result.
\begin{theorem}\label{thm:linear}
Grant Assumptions \ref{ass:smallball:vola} and \ref{ass:linear}. Then both Assumptions \ref{ass:dependence:main} and \ref{ass:smallball:I} hold. In particular, Theorem \ref{thm:edge:smallball} and Corollary \ref{cor:wasser:small:ball} apply.
\end{theorem}

\subsection{Volterra processes}\label{sec:volterra}

In the study of nonlinear processes, Volterra processes are of
fundamental importance, see for instance~\cite{bendat:1990},~\cite{priestley:1988} or ~\cite{rugh:1981}. We consider
\begin{eqnarray*}
&& V_k = \sum_{i = 1}^{\infty} \sum
_{0 \leq j_1 <\cdots<j_i} a_k (j_1,\ldots,
j_i ) \epsilon_{k - j_1} \cdots\epsilon_{k-j_i},
\end{eqnarray*}
where $ \|\epsilon_k \|_p < \infty$ for $p \geq 2$, and $a_k$ are called the $k$-th Volterra kernel. Let
\begin{eqnarray*}
&& A_{k,i} = \sum_{k \in\{j_1,\ldots,j_i\}, 0 \leq j_1 <\cdots<j_i} \bigl|a_k
(j_1,\ldots, j_i ) \bigr|.
\end{eqnarray*}
Then by the triangle inequality, there exists a constant $C$ such that
\begin{eqnarray*}
&& \bigl\|V_k - V_k^* \bigr\|_p \leq C \sum
_{i = 1}^{\infty} \| \epsilon_0
\|_p^{i} \sum_{l \geq k} A_{l,i}.
\end{eqnarray*}

We thus require the following assumption.

\begin{assumption}\label{ass:volterra}
Let $\E |\varepsilon|^{2 \vee q} < \infty$, $q > \max\{3(\alpha + \beta)\}$ ($\alpha, \beta$ as in \eqref{eq:hoelder:function}), such that
\begin{align*}
\sum_{k \geq 1} k^2 \Big(\sum
_{i = 1}^{\infty} \| \epsilon_0
\|_q^{i} \sum_{l \geq k} A_{l,i} \Big)^{\beta} < \infty.
\end{align*}
\end{assumption}

We then have the following result.
\begin{theorem}\label{thm:volterra}
Grant Assumptions \ref{ass:smallball:vola} and \ref{ass:volterra}. Then both Assumptions \ref{ass:dependence:main} and \ref{ass:smallball:I} hold. In particular, Theorem \ref{thm:edge:smallball} and Corollary \ref{cor:wasser:small:ball} apply.
\end{theorem}

\section{Proofs of Theorem \ref{thm:edge:smallball} and Theorem \ref{thm:wasser:general}}

Throughout the proof, for notational convenience, we assume for simplicity that we are in the time-homogenous Bernoulli-shift case. This means that $g_k = g$ in \eqref{eq_structure_condition}, and, in particular, the quantities in Assumption \ref{ass:smallball:I} are invariant in $l \in \Z$.\\
\\
For $0 \leq a \leq b$, \hypertarget{TTab:eq5}{define} the \textit{Berry-Esseen tail}
\begin{align}
\Td_{a}^{b}(x) = \int_{a \leq |\xi| \leq b} e^{-\ic \xi x}\E\bigl[e^{\ic \xi S_n/\sqrt{n}} \bigr] \Bigl(1 - \frac{|\xi|}{b}\Bigr) \frac{1}{\xi} d \, \xi,
\end{align}
which arises naturally in Berry's smoothing inequality. Using $\Td_{a}^{b}$, for $a > 0$, we consider \hypertarget{ca:eq5}{the} \textit{Berry-Esseen characteristic}
\begin{align}\label{berry_esseen_char}
\Cd_{a} = \inf_{b \geq a}\Big(\sup_{x \in \R}\bigl|\Td_{a}^{b}(x) \bigr| + 1/b\Big).
\end{align}
For $m \in \N$, define the following $\sigma$-algebra
\begin{align}\label{defn_sigma_algebra}
\FF_m = \sigma\bigl(\epsilon_{-m+1},...,\epsilon_0,\epsilon_{m+1},...,\epsilon_{2m},\epsilon_{3m+1},...\bigr).
\end{align}
Moreover, for $1 \leq j \leq n$, let $(\epsilon_k^{(j)})_{k \in \Z}$ be independent copies of $(\epsilon_k)_{k \in \Z}$. For each $2(j-1)m + 1 \leq k \leq 2jm$, define
\begin{align}
X_{km} = f(\epsilon_k, \epsilon_{k-1}, \ldots, \epsilon_{k-m+1}, \epsilon_{k-m}^{(j)}, \epsilon_{k-m-1}^{(j)}, \ldots),
\end{align}
and note that $X_k \stackrel{d}{=} X_{km}$. Finally, let us introduce the quantities
\begin{align*}
A_{i:j} = \sum_{k = i}^j X_{km}, \quad B_j = A_{2(j-1)m + 1 : 2jm} =  \sum_{k = 2(j-1)m + 1}^{2jm} X_{km}, \quad S_{nm} = \sum_{k = 1}^n X_{km}.
\end{align*}

We recall parts of Theorem 2.7 in ~\cite{jirak:edge:trans:2021}, which we restate as the following lemma for the sake of reference.

\begin{lemma}\label{lem_edge}
Grant Assumption \ref{ass:dependence:main}. Then there exists $\delta > 0$, such that
\begin{align*}
\sup_{x \in \R}\big|\P\big(S_n \leq x\sqrt{n}\big) - \Psi_n(x)\big| \lesssim n^{-\frac{1}{2} - \delta} + \Cd_{T_n},
\end{align*}
where $T_n \geq c \sqrt{n}$ for some $c > 0$.
\end{lemma}

In addition, we require the following technical result in the sequel.
\begin{lemma}\label{lem:bound:Snm:difference}
Grant Assumption \ref{ass:dependence:main}. Then there exists $C > 0$, such that
\begin{align*}
\big\|S_n - S_{nm}\big\|_1 \leq C \sqrt{n} m^{-2}.
\end{align*}
\end{lemma}

\begin{proof}[Proof of Lemma \ref{lem:bound:Snm:difference}]
This is an easy consequence of Equation (50) in~\cite{jirak2020berryesseen} (note $\sum_{k \geq 1} k^{\mathfrak{a}} \|X_k-X_k^{(l,')}\|_p \lesssim \sum_{k \geq 1} k^{\mathfrak{a}} \|X_k-X_k^{(l,*)}\|_p$, $\mathfrak{a} \geq 0$, $p \geq 1$, see also Theorem 1 in~\cite{Wu_fuk_nagaev}).
Note that the construction of $X_{km}$ is slightly different in~\cite{jirak2020berryesseen}, but the argument remains equally valid.
\end{proof}

\begin{proof}[Proof of Theorem \ref{thm:edge:smallball}]
Let $\overline{n} = \lfloor n/m \rfloor$ for $m \in \N$. The proof works with any choice $m \asymp n^{\mathfrak{m}}$, $\mathfrak{m} \in (1/2,1)$. In order to establish the claim, it suffices to show that for any $a > 0$, the Berry-Esseen Characteristic $\Cd_{a \sqrt{n}}$ (cf. \ref{berry_esseen_char}) satisfies
\begin{align}
\Cd_{a \sqrt{n}} \lesssim n^{-1/2 -\delta}, \quad \delta > 0.
\end{align}

To this end, we will study $\E e^{\ic \xi {S}_{n}}$ more closely, subject to Assumptions \ref{ass:dependence:main} and \ref{ass:smallball:I}. Due to $|e^{\ic x}- 1| \leq |x|$, $|e^{\ic x}| = 1$, and Lemma \ref{lem:bound:Snm:difference}, we have

\begin{align} \nonumber \label{eq:thm:edge:smallball:0}
\bigl|\E e^{\ic \xi {S}_{n}} - \E e^{\ic \xi {S}_{nm}}\bigr| &\leq \big|\xi \big| \big\|S_n - S_{nm}\big\|_1\\&\lesssim \big|\xi \sqrt{n} m^{-2}\big| \lesssim  \big|\xi\big| n^{-1 - \delta}, \quad \delta > 0. 
\end{align}

Observe that $(B_j)_{1 \leq j \leq \overline{n}}$ is conditionally independent with respect to $\FF_m$, and is a one-dependent sequence in general. Let $\mathcal{I} = \{1,3,\ldots, 2\lfloor \overline{n}/2 \rfloor - 1 \}$. Then $(B_j)_{j \in \mathcal{I}}$ is an independent sequence. Hence using $|e^{\ic x}| = 1$, we have
\begin{align} \nonumber
\bigl\|\E_{\FF_m} e^{\ic \xi {S}_{nm}} \bigr\|_1 & \leq  \big\|\prod_{j = 1}^{\overline{n}} \big| \E_{\FF_m} e^{\ic \xi {B}_{j}} \big| \big\|_1\\& \leq  \bigl\|\prod_{j \in \mathcal{I}} \bigl|\E_{\FF_m} e^{\ic \xi {B}_{j}}\bigr|\bigr\|_1 = \prod_{j \in \mathcal{I}}\bigl\|\E_{\FF_m} e^{\ic \xi {B}_{j}}\bigr\|_1. \label{eq:zero:point:five}
\end{align}

Next, put $\mathcal{H} = \sigma\big(\mathcal{E}_0 \cup \mathcal{E}_{m+1}' \cup \mathcal{E}_{m+1}^+ \big)$, and observe that for $1 \leq j \leq \overline{n}$, we have the identity
\begin{align*}
\bigl\|\E_{\FF_m} e^{\ic \xi {B}_{j}}\bigr\|_1 = \bigl\|\E_{\mathcal{H}} e^{\ic \xi {B}_{1}}\bigr\|_1.
\end{align*}

Let $A^+ \in \mathcal{G}_{m+1}$. Since $|e^{\ic x}| = 1$, it follows that
\begin{align*}
&\big|\E_{\mathcal{H}} e^{\ic \xi {B}_{1}}\big| \leq \big|\E_{\mathcal{H}} e^{\ic \xi ({B}_{1} - A^+)}\big| \\& \leq
\bigl|\E_{\mathcal{H}} e^{\ic \xi A_{1:m}}\bigr| + \big| \E_{\mathcal{H}}e^{\ic \xi A_{1:m}} (e^{\ic \xi (A_{m+1:2m} - A^+)} - 1) \big|\\&\leq
\bigl|\E_{\mathcal{H}} e^{\ic \xi A_{1:m}}\bigr| + \E_{\mathcal{H}}\big|e^{\ic \xi (A_{m+1:2m} - A^+)} - 1\big|.
\end{align*}
Observe that $A_{m+1:2m} - A^+$ is independent of $\mathcal{E}_0$. Using $|e^{\ic x} - 1| \leq |x|$, we thus further obtain
\begin{align*}
&\E_{\mathcal{H}}\big|e^{\ic \xi (A_{m+1:2m} - A^+)} - 1\big| \leq \bigl|\xi\bigr| \E_{\mathcal{G}_{m+1}}\big| A_{m+1:2m} - A^+\bigr|.
\end{align*}

Then due to \hyperref[B1]{\Bone}, for any $\delta > 0$, there exists a set $\mathcal{A}_{\delta} \in \mathcal{G}_{m+1}$ with $\P(\mathcal{A}_{\delta}) \geq c_{\delta} > 0$, such that for appropriate choice of $A^+$
\begin{align}
\E_{\mathcal{G}_{m+1}}\big|A_{m+1:2m} - A^+\bigr|\ind_{\mathcal{A}_{\delta}} \leq \delta.
\end{align}

Since $\mathcal{A}_{\delta} \in \mathcal{H}$, we obtain from the above and $|e^{\ic x}| = 1$
\begin{align*}
\big|\E_{\mathcal{H}} e^{\ic \xi {B}_{1}}\big| &\leq \big|\E_{\mathcal{H}} e^{\ic \xi {B}_{1}} \ind_{\mathcal{A}_{\delta} }\big| + \E_{\mathcal{H}} \ind_{\mathcal{A}_{\delta}^c} \\& \leq  \Big(\bigl|\E_{\mathcal{H}} e^{\ic \xi A_{1:m}} \bigr| + \delta|\xi|\Big)\ind_{\mathcal{A}_{\delta}} + \ind_{\mathcal{A}_{\delta}^c}.
\end{align*}
Next, observe
\begin{align*}
\E_{\mathcal{H}} e^{\ic \xi A_{1:m}} = \E_{\mathcal{E}_0} e^{\ic \xi A_{1:m}}.
\end{align*}
Since $\mathcal{A}_{\delta}$ is independent from $\mathcal{E}_0$, we get
\begin{align}  \nonumber
\E\big|\E_{\mathcal{H}} e^{\ic \xi {B}_{1}}\big| &\leq \E \Big(\bigl|\E_{\mathcal{E}_0} e^{\ic \xi A_{1:m}} \bigr| + \delta |\xi|\Big)\ind_{\mathcal{A}_{\delta}} +  \E \ind_{\mathcal{A}_{\delta}^c}\\&=  \Big( \E \bigl|\E_{\mathcal{E}_0} e^{\ic \xi A_{1:m}} \bigr|
+ \delta |\xi| \Big) \P\big(\mathcal{A}_{\delta}\big) + \P\big(\mathcal{A}_{\delta}^c\big). \label{eq:one}
\end{align}

We next deal with $\E \bigl|\E_{\mathcal{E}_0} e^{\ic \xi A_{1:m}} \bigr|$. To this end, let $\mathcal{E}_m^{(1)} = \sigma(\epsilon_k^{(1)},\, k \leq m)$. Since $X_{km} \stackrel{d}{=} X_k$ and $|e^{\ic x}| = 1$, we have
\begin{align*}
\E\big|\E_{\mathcal{E}_{0}} e^{\mathrm{i} \xi A_{1:m}}  \big| &\leq \E\big|\E_{\sigma(\mathcal{E}_{m-1}, \mathcal{E}_{m-1}^{(1)})} e^{\mathrm{i} \xi A_{1:m}}  \big| \\&\leq \E\big|\E_{\mathcal{E}_{-1}} e^{\mathrm{i} \xi X_0 } \big|.
\end{align*}

One readily shows that the map $g:\R \to [0,1]$, given by
\begin{align*}
g\big(\xi\big) = \E\big|\E_{\mathcal{E}_{-1}} e^{\mathrm{i} \xi X_0}  \big|,
\end{align*}
is continuous. Let $I = [a,b]$, $0 < a < b$. Then by compactness, there exists $\xi^{\ast} \in I$ such that
\begin{align*}
\sup_{\xi \in I}g\big(\xi\big) = g\big(\xi^{\ast}\big) < 1,
\end{align*}
where we used \hyperref[B2]{\Btwo} for the last inequality. Hence for any $0 <a < b$, there exists $\eta_{ab} < 1$ such that

\begin{align}\label{eq:two}
\sup_{\xi \in [a,b]} \E\big|\E_{\mathcal{E}_{-1}} e^{\mathrm{i} \xi X_0}  \big| \leq \eta_{ab} < 1.
\end{align}

Setting $\delta = (1 - \eta_{ab})/2b$, we obtain from \eqref{eq:one} and \eqref{eq:two} that for any $\xi \in [a,b]$, there exists $\rho_{ab} < 1$ such that (recall $\P(\mathcal{A}_{\delta}) \geq c_{\delta}$)

\begin{align*}
\E\big|\E_{\mathcal{H}} e^{\ic \xi {B}_{1}}\big| &\leq  \Big( \E \bigl|\E_{\mathcal{E}_0} e^{\ic \xi A_{1:m}} \bigr|
+ \delta |\xi| \Big) \P\big(\mathcal{A}_{\delta}\big) + \P\big(\mathcal{A}_{\delta}^c\big)
\\& \leq \frac{1 + \eta_{ab}}{2} \P\big(\mathcal{A}_{\delta}\big) + \P\big(\mathcal{A}_{\delta}^c\big) \leq \rho_{ab}.
\end{align*}
Consequently, since $|\mathcal{I}| \geq \overline{n}/3$ for $n$ large enough, we get
\begin{align*}
\sup_{\xi \in [a,b]} \prod_{j \in \mathcal{I}}\bigl\|\E_{\FF_m} e^{\ic \xi {B}_{j}}\bigr\|_1 \leq  \rho_{ab}^{\frac{\overline{n}}{3}}.
\end{align*}
Combining \eqref{eq:zero:point:five} and the above, we conclude that there exists $b_n \to \infty$ such that for any $a > 0$, we have
\begin{align*}
\sup_{\xi \in [a,b_n]} \prod_{j \in \mathcal{I}}\bigl\|\E_{\FF_m} e^{\ic \xi {B}_{j}}\bigr\|_1 \lesssim n^{-1/2 - \delta'}, \quad \delta' > 0.
\end{align*}
Using \eqref{eq:thm:edge:smallball:0}, it follows that
\begin{align}
\sup_{\xi \in [a,b_n]} \Big|\E e^{\ic \xi {S}_{n}} \Big| \lesssim  n^{-1/2 - \delta'},
\end{align}
and
\begin{align*}
\Cd_{a \sqrt{n}} \leq \big(\sup_{x \in \R}\bigl|\Td_{a \sqrt{n}}^{b_n \sqrt{n}}(x) \bigr| + 1/(b_n \sqrt{n})\big) = o\big(n^{-1/2}\big),
\end{align*}
which, by virtue of Lemma \ref{lem_edge}, completes the proof.
\end{proof}

For the proof of Theorem \ref{thm:wasser:general}, we require some additional notation. For $e > 0$ and $f \in \N$ even, let $G_{e,f}$  be a real valued random variable with density function
\begin{align}\label{eq_g_ab}
g_{e,f}(x) = c_f e \Big|\frac{\sin(e x)}{e x} \Big|^f, \quad x \in \R,
\end{align}
for some constant $c_f > 0$ only depending on $f$. It is well-known (cf. ~\cite{Bhattacharya_rao_1976_reprint_2010}, Section 10) that for even $f$ the Fourier transform $\hat{g}_{e,f}$ {satisfies}
\begin{align}\label{eq_thm_smooth_fourier}
\hat{g}_{e,f}(t) =  \left\{
\begin{array}{ll}
2 \pi c_f u^{\ast \, f}[-e,e](t) &\text{if $|t| \leq e f$},\\
0 &\text{otherwise},
\end{array}
\right.
\end{align}
where $u^{\ast \, f}[-e,e]$ denotes the $f$-fold convolution of {the} density of the uniform distribution on $[-e,e]$, that is $u[-e,e](t) = \frac{1}{2e} \ind_{[-e,e]}(t)$. For $f \geq 6$, let $(H_k)_{k \in \Z}$ be i.i.d. with $H_k \stackrel{d}{=} G_{e,f}$ and independent of $S_n$. For $\eta > 0$, {define}
\begin{align}\label{defn_diamond_mod}
{X}_k^{\diamond} = X_k + \eta H_{k} - \eta H_{k-1}, \quad {S}_n^{\diamond} = \sum_{k = 1}^n {X}_k^{\diamond} = S_n + \eta H_n - \eta H_0,
\end{align}
and $L_n^{\diamond}$ in analogy.

\begin{proof}[Proof of Theorem \ref{thm:wasser:general}]
Using standard properties of the Wasserstein distance and the triangle inequality, we arrive at
\begin{align}\label{eq:thm:gen:w1:1}
W_1\big(\P_{S_n^{}/\sqrt{n}}, \P_{L_n}\big) \leq W_1\big(\P_{S_n^{\diamond}/\sqrt{n}}, \P_{L_n^{\diamond}}\big) + 4 \eta \E\big|G_{e,f}\big|/\sqrt{n}.
\end{align}
For $f \geq 6$ and small enough $e> 0$, we get, using \eqref{eq:thm:edge:smallball:condition}, for any $c,\eta > 0$
\begin{align}\label{eq:thm:gen:w1:1.5}
\sup_{|\xi| \geq c }\big|\E e^{\ic \xi S_n^{\diamond}} \big| \leq  \sup_{c \leq |\xi| \leq ef/\eta}\big|\E e^{\ic \xi S_n}\big| \leq C_{c,\delta} n^{-1/2-\delta}, \quad \delta > 0.
\end{align}
Following the proof of Theorem 3.6. in ~\cite{jirak:edge:trans:2021}, we obtain
\begin{align}\label{eq:thm:gen:w1:1.7}
W_1\big(\P_{S_n^{\diamond}/\sqrt{n}}, \P_{L_n^{\diamond}}\big) \lesssim  n^{-p/2 +1} + \int_{|x| \leq \tau_n} \big|\P(S_n^{\diamond} \leq x \sqrt{n}) - \P(L_n^{\diamond} \leq x )\big| dx,
\end{align}
where $\tau_n \lesssim \sqrt{\log n}$. By \eqref{eq:thm:gen:w1:1.5}, we have for $\Td_{a}^{b}(x)$ (defined with respect to $S_n^{\diamond}$)
\begin{align*}
\big|\Td_{a \sqrt{n}}^{b \sqrt{n}}(x)\big| \leq \int_{a \leq |\xi|/\sqrt{n} \leq b} \Big| e^{-\ic \xi x}\E e^{\ic \xi S_n^{\diamond}/\sqrt{n}} \Bigl(1 - \frac{|\xi|}{b \sqrt{n}}\Bigr) \frac{1}{\xi}\Big|  d \, \xi \leq C_{a,\delta} \frac{\log(n b)}{n^{-1/2 - \delta}},
\end{align*}
which does not depend on $x$. Hence for $\Cd_{a\sqrt{n}}$ (defined with respect to $S_n^{\diamond}$)
\begin{align*}
\Cd_{a \sqrt{n}} = \inf_{b \geq a}\Big(\sup_{x \in \R}\bigl|\Td_{a \sqrt{n}}^{b \sqrt{n}}(x) \bigr| + 1/(b \sqrt{n})\Big) \leq C_{a,\delta}' \frac{\log(n)}{n^{-1/2 - \delta}}.
\end{align*}
An application of Lemma \ref{lem_edge} then yields
\begin{align*}
\sup_{x \in \R}\big|\P(S_n^{\diamond} \leq x \sqrt{n}) - \P(L_n^{\diamond} \leq x ) \big| \leq C_{a,\delta}'' n^{-1/2-\delta'}, \quad \delta' > 0.
\end{align*}
Plugging this into \eqref{eq:thm:gen:w1:1.7}, we obtain
\begin{align}\label{eq:thm:gen:w1:2}
W_1\big(\P_{S_n^{\diamond}/\sqrt{n}}, \P_{L_n^{\diamond}}\big) \leq C_{\delta,\delta}''' n^{-1/2 - \delta'}\big(1 + \log n\big).
\end{align}
Selecting $\eta = \eta_n \to 0$ sufficiently slow (\eqref{eq:thm:edge:smallball:condition} must be valid, see \eqref{eq:thm:gen:w1:1.5}), the claim follows by combining \eqref{eq:thm:gen:w1:1} and \eqref{eq:thm:gen:w1:2}.
\end{proof}

\section{Proofs of Volatility models}

We first state the following elementary lemma.

\begin{lemma}\label{lem:hoelder}
Suppose that the function $f$ satisfies \eqref{eq:hoelder:function}, and assume $\|Y_0\|_{p(\alpha + \beta)} + \sum_{k \in \N} k^2 \|Y_k-Y_k^*\|_{p(\alpha + \beta)}^{\beta} < \infty$, $p(\alpha + \beta)$, for $p \geq 3$. Then $X_k = f(Y_k) - \E f(Y_k)$ satisfies \hyperref[A2]{\Atwo}.
\end{lemma}

\begin{proof}[Proof of Lemma \ref{lem:hoelder}]
Using Hölders inequality with $r =\alpha/\beta+1$, $s = (\alpha + \beta)/\alpha$, we get
\begin{align*}
\big\|X_k - X_k^*\big\|_p &\leq \big\||Y_k - Y_k^*|^{\beta} (L + |Y_k|^{\alpha} + |Y_k^*|^{\alpha}\big\|_p \leq
\big\|Y_k - Y_k^*\big\|_{rp\beta}^{\beta}\big(L + 2\big\|Y_k\big\|_{sp\alpha}^{\alpha} \big)
\\&= \big\|Y_k - Y_k^*\big\|_{p(\alpha + \beta)}^{\beta}\big(L + 2\big\|Y_k\big\|_{p(\alpha + \beta)}^{\alpha}\big).
\end{align*}
\end{proof}

\subsection{Proof of Theorem \ref{thm:garch}}

\begin{proof}[Proof of Theorem \ref{thm:garch}]
In order to apply Theorem \ref{thm:edge:smallball}, we need to validate both Assumptions \ref{ass:dependence:main} and \ref{ass:smallball:I}, based on Assumptions \ref{ass:smallball:vola} and \ref{ass:garch}. We will do so below. Since, as mentioned above, $V_k$ is a time-homogenous Bernoulli-shift process, the quantities in Assumption \ref{ass:smallball:I} do not depend on $l \in \Z$, simplifying the notation. We first consider the case $h_n \equiv 0$.
\\
\hyperref[B1]{\Bone}: We first validate \hyperref[B1]{\Bone}, which requires most attention. To this end, we first introduce some necessary quantities. Let
\begin{align*}
(V_k^+)^{2\lambda} =  \sum_{m = 1}^{\infty} &\sum_{1 \leq l_1,\ldots,l_m\leq \mathfrak{r}}  g_{l_m}(\varepsilon_{k - l_1 - \ldots - l_m})\ind(k>l_1 + \ldots + l_m) \\& \times \prod_{i = 1}^{m-1} c_i(\varepsilon_{k - l_1 - \ldots - l_i})\ind(k>l_1 + \ldots + l_i),
\end{align*}
and $X_k^+ = f(\varepsilon_k V_k^+) \in \mathcal{E}_1^+$ for $k \geq 1$ (note that $V_k \in \mathcal{E}_{k-1}$ here by construction). Observe the bound
\begin{align}\label{eq:V-V^+:bound}
\big|(V_k^+)^{2\lambda} - V_k^{2\lambda}\big| &\leq \sum_{m > \lfloor k/\mathfrak{r} \rfloor}^{\infty} \sum_{1 \leq l_1,\ldots,l_m\leq \mathfrak{r}}  g_{l_m}(\varepsilon_{k - l_1 - \ldots - l_m}) \prod_{i = 1}^{m-1} c_i(\varepsilon_{k - l_1 - \ldots - l_i}),
\end{align}
which we will repeatedly use. For $\delta > 0$, denote with
\begin{align*}
\mathcal{A}_{1j\delta} = \big\{\sum_{k = 1}^j \E_{\mathcal{E}_1^+}|X_k - X_k^+| \leq \delta \big\},
\end{align*}
and let $\mathcal{E}_{ij} = \sigma(\varepsilon_k,\, i \leq k \leq j)$ (note $\mathcal{E}_{1\infty} = \mathcal{E}_1^+$). Then
\begin{align}\label{eq:lower:bound:key} \nonumber
\P\Big(\sum_{k = 1}^{\infty} \E_{\mathcal{E}_1^+}|X_k - X_k^+| \leq 2\delta \Big) &\geq \P\Big(\sum_{k > j}^{\infty} \E_{\mathcal{E}_1^+}|X_k - X_k^+| \leq \delta \cap \mathcal{A}_{1j\delta} \Big) \\&= \nonumber \E \P_{\mathcal{E}_{1j}}\Big(\sum_{k > j}^{\infty} \E_{\mathcal{E}_1^+}|X_k - X_k^+| \leq \delta \Big) \ind_{\mathcal{A}_{1j\delta}} \\&\geq \E \Big(1 - \delta^{-1}\sum_{k > j}^{\infty} \E_{\mathcal{E}_{1j}}|X_k - X_k^+| \Big) \ind_{\mathcal{A}_{1j\delta}},
\end{align}
where we used Markovs inequality in the last step. This simple lower bound is the key for establishing \hyperref[B1]{\Bone}.

Let $\mathcal{B}_{1j\eta} = \{|\varepsilon_k| \leq \eta, \, 1 \leq k \leq j\}$. Select $0 < \eta \leq 1$ such that $\sup_{|\varepsilon| \leq \eta}c_{i}(\varepsilon) \leq \|c_i(\varepsilon_0)\|_{q}$ for $1 \leq i \leq \mathfrak{r}$, which is always possible by Assumption \ref{ass:garch}. Assume without loss of generality $\|g_i(\varepsilon)\|_q \leq C_g$. Then by Assumption \ref{ass:garch}, we have on the event $\mathcal{B}_{1j\eta}$
\begin{align*}
\Big(\E_{\mathcal{E}_{1j}} \Big| \sum_{1 \leq l_1,\ldots,l_m\leq \mathfrak{r}} g_{l_m}(\varepsilon_{k - l_1 - \ldots - l_m}) \prod_{i = 1}^{m-1} c_i(\varepsilon_{k - l_1 - \ldots - l_i}) \Big|^q\Big)^{1/q} \leq C_g \rd \gamma_c^{m-1},
\end{align*}
where we recall $\|c_1(\varepsilon)\|_q + \ldots \|c_{\mathfrak{r}}(\varepsilon) \|_q \leq \gamma_c < 1$. From the above, we conclude (on the event $\mathcal{B}_{1j\eta}$),
\begin{align}\label{eq:garch:V:diff}
\big(\E_{\mathcal{E}_{1j}}\big|V_k^{2\lambda} - (V_k^+)^{2\lambda}\big|^{q}\big)^{1/q}\leq  C_g \rd  \frac{\gamma_c^{\lfloor k/\rd \rfloor}}{1 - \gamma_c},
\end{align}
and similarly, one obtains (still on the event $\mathcal{B}_{1j\eta}$)
\begin{align}\label{eq:garch:V:bound}
\big(\E_{\mathcal{E}_{1j}}\big|V_k^{2\lambda}\big|^{q}\big)^{1/q} \leq C_g \rd  \frac{1}{1 - \gamma_c}.
\end{align}

In the following derivations below, all norms $\|\cdot\|_r$ are taken with respect to $\P_{\mathcal{E}_{1j}}$. Since $V_k \geq V_k^+$ and $\lambda \geq 1/2$, we have from $|x^{1/(2\lambda)}-y^{1/(2\lambda)}| \leq |x-y|^{1/(2\lambda)}$ and Cauchy-Schwarz
\begin{align}\label{eq:garch:3} \nonumber
\big\|X_k - X_k^+\big\|_1 &\leq \big\| |\varepsilon_k|^{\beta} |V_k - V_k^+|^{\beta}\big(1 + |2\varepsilon_k V_k|^{\alpha}\big\|_1 \big) \\& \nonumber \leq \big\||\varepsilon_k|^{\beta} |V_k^{2\lambda} - (V_k^+)^{2\lambda}|^{\beta/(2\lambda)}\big(1 + |2\varepsilon_k V_k|^{\alpha} \big\|_1 \big) \\&\leq
\big\| |\varepsilon_k|^{\beta} |V_k^{2\lambda} - (V_k^+)^{2\lambda}|^{\beta/(2\lambda)} \big\|_2
\big\|1 + |2\varepsilon_k V_k|^{\alpha}\big\|_2.
\end{align}

By independence, \eqref{eq:garch:V:diff} and Jensens inequality (if ${\beta/\lambda} < 1$, we again apply Jensens inequality in addition)
\begin{align}\label{eq:garch:4}
\big\| |\varepsilon_k|^{2\beta} |V_k^{2\lambda} - (V_k^+)^{2\lambda}|^{\beta/\lambda} \big\|_1^{\lambda/\beta} \leq \big\| \varepsilon_k^2\big\|_1^{\lambda} C_g \rd  \frac{\gamma_c^{\lfloor k/\rd \rfloor}}{1 - \gamma_c}.
\end{align}
Similarly, by independence and \eqref{eq:garch:V:bound} 
\begin{align}\label{eq:garch:5}
\big\|1 + |2\varepsilon_k V_k|^{\alpha}\big\|_2 \leq 1 + 2\big\|\varepsilon_k\big\|_{2\alpha}^{\alpha}  \Big(C_g \rd  \frac{1}{1 - \gamma_c}\Big)^{\frac{\alpha}{2\lambda}}.
\end{align}

All in all, on the event $\mathcal{B}_{1j\eta}$, combining \eqref{eq:garch:3}, \eqref{eq:garch:4}, and \eqref{eq:garch:5}, we arrive at
\begin{align*}
\E_{\mathcal{E}_{1j}}|X_k - X_k^+\big|_1 \leq C_+ \rho^k \big(\E_{\mathcal{E}_{1j}} |\varepsilon_k|^2\big)^{\beta/2},
\end{align*}
where $C_+$ does not depend on $\eta$, and $\rho < 1$ only depends on $\gamma_c$, $\lambda$ and $\beta$. Moreover, we have $\E_{\mathcal{E}_{1j}} \varepsilon_k^2 = \varepsilon_k^{2} \leq \eta^{2}$ for $1 \leq k \leq j$ on the event $\mathcal{B}_{1j\eta}$, and consequently
\begin{align}\label{eq:garch:6}
\sum_{k = 1}^j \E_{\mathcal{E}_{1j}}|X_k - X_k^+| \leq C_{+} \eta^{\beta} \frac{1}{1 - \rho}.
\end{align}
Hence selecting $\eta$ such that $C_{+} \eta^{\beta} \frac{1}{1 - \rho} < \delta$, we conclude $\mathcal{B}_{1j\eta} \subseteq \mathcal{A}_{1j\delta}$, and \hyperref[V1]{\Vone} yields
\begin{align}
\P\big(\mathcal{A}_{1j\delta}\big) \geq \P\big(\mathcal{B}_{1j\eta}\big) > 0.
\end{align}

Similarly, we obtain on $\mathcal{B}_{1j\eta}$, for $k > j$, the estimate

\begin{align*}
\sum_{k > j} \E_{\mathcal{E}_{1j}}|X_k - X_k^+| \leq C_{+} \frac{\rho^j}{1 - \rho} \big(\E |\varepsilon_0|^2\big)^{\beta/2}.
\end{align*}

Selecting $j_0$ sufficiently large, we get for $j \geq j_0$
\begin{align*}
\E \Big(1 - \delta^{-1}\sum_{k > j}^{\infty} &\E_{\mathcal{E}_{1j}}|X_k - X_k^+| \Big) \ind_{\mathcal{A}_{1j\delta}} \\&\geq
\E \Big(1 - \delta^{-1} C_{+} \frac{\rho^j}{1 - \rho} \big(\E |\varepsilon_0|^2\big)^{\beta/2}\Big) \ind_{\mathcal{B}_{1j\delta}}\\&\geq \E \Big(1 - \frac{1}{2}\Big) \ind_{\mathcal{B}_{1j\delta}} = \frac{\P(\mathcal{B}_{1j\eta}) }{2} > 0.
\end{align*}

Hence \hyperref[B1]{\Bone} holds. Let us now consider the case $h_n \not \equiv 0$, which turns out to be just a minor extension. Arguing as above in \eqref{eq:garch:3}, \eqref{eq:garch:4}, and \eqref{eq:garch:5}, we obtain
$\sum_{k = 1}^{\infty} \|h_n(V_k) - h_n(V_k^+)\|_1 \leq C$, where $C$ does not depend on $n$. Hence by the triangle and Jensen's inequality
\begin{align}\label{eq:garch:7}
\Big\|\sum_{k = 1}^{\infty} \E_{\mathcal{E}_1^+}\big|h_n(V_k) - h_n(V_k^+)\big|\Big\|_1 \leq C.
\end{align}
On the other hand, \eqref{eq:h:limit:zero}, the triangle and Jensen's inequality imply that there exists $l_n \to \infty$, such that
\begin{align}
\Big\|\sum_{k = 1}^{l_n} \E_{\mathcal{E}_1^+}\big|h_n(V_k)\big|\Big\|_1 \leq  \sum_{k = 1}^{l_n} \E \big|h_n(V_k)\big| \to 0
\end{align}
as $n$ increases. Setting $h_n^+(V_k) = h_n(V_k^+)$ for $k \geq l_n$ and $h_n^+(V_k) \equiv 0$ otherwise, we conclude from the above
\begin{align}
\lim_{n \to \infty} \Big\|\sum_{k = 1}^{\infty} \E_{\mathcal{E}_1^+}\big|h_n(V_k) - h_n^+(V_k)\big|\Big\|_1 = 0.
\end{align}
Piecing everything together, the validity of \hyperref[B1]{\Bone} follows.\\

\hyperref[B2]{\Btwo}:
Note that $v \mapsto \E e^{\mathrm{i} \xi f(v \epsilon) } $ is continuous in $v$. Then by compactness of $\mathcal{V}$ and
\hyperref[V2]{\Vtwo}, we have $\sup_{v \in \mathcal{V}}\big|\E_{} e^{\mathrm{i} \xi f(v \epsilon) }  \big| \ind_{\mathcal{V}} < 1$, and due to $|e^{\mathrm{i} z}| = 1$, $h_n(V_k) \in \mathcal{E}_{k-1}$ and \hyperref[V2]{\Vtwo}, we thus conclude
\begin{align*}
\E\big|\E_{\mathcal{E}_{-1}} e^{\mathrm{i} \xi X_0 } \big| &\leq \E \sup_{v \in \mathcal{V}}\big|\E_{} e^{\mathrm{i} \xi f(v \epsilon) }  \big| \ind_{\mathcal{V}} + \P\big(\mathcal{V}^c\big) \\&< \P\big(\mathcal{V}\big) + \P\big(\mathcal{V}^c\big) = 1.
\end{align*}
Hence \hyperref[B2]{\Btwo} holds.\\

\hyperref[A2]{\Atwo}:
Arguing similarly as above for establishing \hyperref[B1]{\Bone}, one derives
\begin{align}
\|\varepsilon_k V_k - \varepsilon_k V_k^{\ast}\|_q, \, \|V_k - V_k^{\ast}\|_q \leq C \rho^k, \quad \rho < 1.
\end{align}
The claim now follows from Lemma \ref{lem:hoelder} and the triangle inequality. \\
\hyperref[A1]{\Aone}: We may repeat arguments employed in \hyperref[A2]{\Atwo} (resp. \hyperref[B1]{\Bone}).
\end{proof}

\subsection{Proof of Theorem \ref{thm:iterated}}

\begin{proof}[Proof of Theorem \ref{thm:iterated}]
Let $V_k^+ = F_{\varepsilon_k} \circ F_{\varepsilon_{k-1}} \circ \ldots \circ F_{\varepsilon_l}(0)$. Then $V_k^+ \in \mathcal{E}_l^+$. Although we are no longer in the time-homogenous Bernoulli-shift setup, it is obvious that we can repeat the proof of Theorem \ref{thm:garch}, almost verbatim. In fact, due to the more explicit iterative structure, some computations are even simpler.
\end{proof}

\subsection{Proof of Theorem \ref{thm:linear}}

\begin{proof}[Proof of Theorem \ref{thm:linear}]
As for augmented Garch sequences, we are again in the time-homogenous Bernoulli-shift case. Let
\begin{align*}
G_k^+ = \sum_{i = 0}^{k-1} a_i c_i(\epsilon_{k-i}), \quad V_k^+ = g(G_k^+).
\end{align*}
Then $V_k^+ \in \mathcal{E}_1^+$. It is again obvious that we can repeat the proof of Theorem \ref{thm:garch}, almost verbatim. As in the case of Theorem \ref{thm:iterated}, the actual proof is even simpler.
\end{proof}

\subsection{Proof of Theorem \ref{thm:volterra}}

\begin{proof}[Proof of Theorem \ref{thm:volterra}]
As for augmented Garch sequences and functions of linear processes, we are again in the time-homogenous Bernoulli-shift case.
Let
\begin{align*}
V_k^+ = \sum_{i = 1}^{\infty} \sum
_{0 \leq j_1 <\cdots<j_i \leq k} a_k (j_1,\ldots,
j_i ) \epsilon_{k - j_1} \cdots\epsilon_{k-j_i}.
\end{align*}
Since clearly $V_1^+ \in \mathcal{E}_1^+$, we may now repeat the proof of Theorem \ref{thm:garch}. As in previous cases, the actual proof is even simpler.
\end{proof}

\end{document}